\documentclass[a4paper,12pt]{amsart}
\usepackage{amssymb}
\usepackage{verbatim}
\usepackage{enumitem}

\setenumerate[1]{label=(\alph{*}), ref=(\alph{*})}
\setenumerate[2]{label=(\roman{*}), ref=(\roman{*})}

\newcommand\enn{\mathbb N}
\newcommand\err{\mathbb R}
\newcommand{\eps}{\varepsilon}
\newcommand{\spann}{\mbox{span}}

\newcommand{\HB}{\text{H{\kern -0.35em}B}}
\newcommand{\vertiii}[1]{{\left\vert\kern-0.25ex\left\vert\kern-0.25ex\left\vert#1\right\vert\kern-0.25ex\right\vert\kern-0.25ex\right\vert}}
\DeclareMathOperator{\supp}{supp}

\newtheorem{thm}{Theorem}[section]
\newtheorem{prop}[thm]{Proposition}
\newtheorem{lem}[thm]{Lemma}
\newtheorem{cor}[thm]{Corollary}

\theoremstyle{definition}

\theoremstyle{remark}

\newtheorem{rem}[thm]{Remark}

\author[T.~A.~Abrahamsen]{Trond A.~Abrahamsen}
\address[T.~A.~Abrahamsen]{Department of Mathematics, University of
  Agder, Postbox 422 4604 Kristiansand, Norway.} 
\email{trond.a.abrahamsen@uia.no}
\urladdr{http://home.uia.no/trondaa/index.php3}

\author[P.~H\'ajek]{Petr H\'ajek}
\address[P.~H\'ajek]{Department of Mathematics\\Faculty of Electrical
Engineering\\Czech Technical University in Prague\\Technick\'a 2, 166 27
Praha 6\\ Czech Republic}
\email{hajek@math.cas.cz}

\author[S.~Troyanski]{Stanimir Troyanski}
\address[S.~Troyanski]{Institute of Mathematics and Informatics,
  Bulgarian Academy of Science, bl.8, acad. G.~Bonchev str.~1113 Sofia, Bulgaria and
  Departamento de Matem\'aticas, Universidad de Murcia, Campus de
  Espinardo, 30100 Espinardo (Murcia), Spain}
\email{stroya@um.es}

\thanks{P.~H\'ajek was supported by OPVVV CAAS
  CZ.02.1.01/0.0/0.0/16$\_$019/0000778.}
\thanks{S. Troyanski was supported  by MTM2017-86182-P (AEI/FEDER, UE), and by the Bulgarian National Scientific Fund, Grant KP–06–H22/4, 04.12.2018.}


\title{Almost square dual Banach spaces}

\begin{document}
\maketitle
 
\begin{abstract}
  We show that finite dimensional Banach spaces fail to be uniformly non
  locally almost square. Moreover, we construct an equivalent almost
  square bidual norm on $\ell_\infty.$ As a consequence we get that
  every dual Banach space containing $c_0$ has an equivalent almost
  square dual norm. Finally we characterize
  separable real almost square spaces in terms of their position in
  their fourth duals. 
\end{abstract}

\section{Introduction}
\label{sec:intro}
A Banach space $X$ with unit sphere $S_X$ is said to be
  \begin{enumerate}
    \item \label{asq}\emph{almost square (ASQ)} if for every $\eps >0$ and every
      finite set $(x_i)_{i=1}^n \subset S_X$ there exists $h \in
      S_X$ such that 
      \[\|x_i \pm h\| \le 1+\eps \hspace{2mm} \text{ for every } i=1,
      \ldots, n.\] 
    \item \label{lasq}\emph{locally almost square (LASQ)} if for every $\eps >0$ and
      every $x \in S_X$ there exists $h \in
      S_X$ such that 
      \[\|x \pm h\| \le 1+\eps.\] 
    \end{enumerate}
These notions were introduced in \cite{MR3415738} and have later on been
considered in several papers, among others \cite{MR3466077},
\cite{MR3619230}, and \cite{MR3661153}. The ASQ property is connected to the
intersection property, introduced by Harmand and Behrends
in \cite{BeHa} when studying proper M-ideals. A Banach space has the
\emph{intersection property (IP)} if for every $\eps > 0$ there are
$(x_i)_{i=1}^n \subset X$ with $\|x_i\| < 1$ for $i=1, \ldots , n$
such that $\|x_i - h\| \le 1$ implies $\|h\| \le \eps.$ The
connection is that a Banach space is ASQ if and only if it
fails the IP for all $0 < \eps < 1$
\cite[Proposition~6.1]{MR3415738}. The main result in the present
paper is that there exists a dual (in fact bidual) ASQ space. Using
this we give a positive answer to a question from
\cite[Question~6.6]{MR3415738} (later posed in
\cite[p.~1039]{MR3466077} and \cite[p.~130]{MR3619230}) whether ASQ
spaces can be dual. Also we give a negative
answer to a question from \cite[p.~167]{BeHa} whether a space that can never be
a proper M-ideal must have the IP. The latter follows
since, as noted in \cite[p.~167]{BeHa}, dual spaces can never be proper M-ideals. 

Let us now provide some background on ASQ and LASQ spaces. A typical example
of an ASQ space is $c_0.$ Just take the $h$ in the definition to be a
standard unit vector $e_n$ with big enough $n.$ The reason why this
works is that functions in $c_0$ have a common
zero at infinity. That of ``having a common zero'' is in some sense
a generical feature of the vectors in ASQ spaces while for LASQ spaces it
is that of ``having a zero''. The latter can be
exemplified by the space $X = \{f \in C[0,1]: f(0) = -f(1)\}$ (As usual
$C[0,1]$ denotes the space of continuous functions on the unit interval
with the max-norm). In $X$ all functions have a zero, but in this
space two different functions can certainly have different zeros, so
$X$ is LASQ, but not ASQ \cite[Example~3.7]{MR3415738}.

Both ASQ and LASQ spaces possess diameter two properties, so they
belong to a class of spaces which, in a sense, is diametrically
opposed to the class of spaces with the Radon-Nikodym property. If a
Banach space $X$ is ASQ, then all
finite convex combinations of slices of the unit ball $B_X$ of $X$
have diameter two \cite[Proposition~2.5 and Theorem~1.3]{MR3415738},
and if $X$ is LASQ then all slices of $B_X$ have diameter two
\cite[Proposition~2.5]{MR3235311}. 
Recall that a closed subspace $X$ of Banach space $Z$ is an M-ideal in
$Z$ provided there exists some Banach space $Y$ such that $Z^* = Y
\oplus_1 X^{\perp}$ where $X^{\perp}$ is the annihilator of $X$ (see
the book \cite{HWW} for the theory of M-ideals). A non-reflexive space
which is an M-ideal in its bidual
(i.e. an M-embedded space) is ASQ \cite[Corollary~4.3]{MR3415738}, so the class of
ASQ spaces is quite big and in fact quite a lot bigger than the class
of non-reflexive M-embedded spaces. Take e.g. $c_0(\ell_1),$ the
$c_0$-sum of $\ell_1.$ This space is ASQ by using the same idea as for
$c_0$, but the space is not M-embedded since the property of being
M-embedded is inherited by subspaces
\cite[III.1.Theorem~1.6]{HWW}. However, because of its $c_0$-like
nature, it is perhaps not so surprising that ASQ spaces happen to
share several properties by non-reflexive M-embedded
spaces. E.g. besides having quite strong diameter two properties, they
contain copies of $c_0$ and are stable by taking
$c_0$-sums \cite[Lemma~2.6 and
  Example~3.1]{MR3415738}. Non-reflexive M-embedded spaces are never
dual spaces, and
the same is true for a large class of ASQ spaces. For example, weakly
compactly generated dual spaces are never ASQ
\cite[p.~1564]{MR3415738}. Moreover, in \cite[Theorem~2.5]{MR3466077}
it was shown that if you put constraints on the vectors $h$ in \ref{asq} in the
definition of an ASQ space, then
the space can never be dual. More precisely it was shown that if a
Banach space $X$ satisfies that for any $\delta > 0$ there exists a
set $(h_\gamma)_{\gamma   \subset \Gamma} \subset S_X$ such that
\begin{enumerate}
\item for every $\eps > 0$ and $x_1, \ldots, x_n \in S_X$ there exists
  $\gamma \in \Gamma$ such that $\|x
  _i \pm h_\gamma\| \le 1 + \eps$ for every $i = 1, \ldots, n,$ and
\item for every finite subset $F$ of $\Gamma$ and any choice of signs
  $\xi_\gamma \in \{1,-1\}$ we have $\|\sum_{\gamma \in F}
  \xi_\gamma h_\gamma\| \le 1 + \delta,$
\end{enumerate}
then it cannot be a dual space. Such a space is said to be
\emph{unconditional almost square (UASQ)}. Among UASQ spaces we find
separable ASQ spaces \cite[Corollary 2.4]{MR3466077}. In
\cite[Question~6.1]{MR3466077} it was left open whether dual ASQ
spaces are UASQ. Our construction of a dual ASQ space answers also
this in the negative.

Let us now describe the content of the paper. In Section
\ref{sec:finitedim-unlasq} we show that finite dimensional Banach
spaces with a fixed dimension are uniformly non LASQ, but we construct
examples showing that this it not so if you allow the dimension to grow.

In Section \ref{sec:bidual-asq} the technique
developed in Section
\ref{sec:finitedim-unlasq}  for finite dimensional spaces is
refined to construct an equivalent bidual ASQ norm on $\ell_\infty.$
This in turn is used to show that every dual Banach space containing
$c_0$ admits an equivalent dual ASQ norm.

In Section \ref{sec:char-sep-asq} we characterize separable real ASQ spaces
in terms of their position in their fourth dual.

\section{Finite dimensional spaces}
\label{sec:finitedim-unlasq}
We start this section by showing that for a fixed dimension, finite dimensional
spaces are uniformly non LASQ.

\begin{prop}
  \label{prop:fixdim-john}
   Given $n \in \enn.$ Then there exists $\delta > 0$ such that for
   all Banach spaces $X$ with dimension $n$, there exists $x \in S_X$
   such that for all $h \in S_X$ we have 
   \begin{align}
     \label{eq:1}
      \max\|x \pm h\| \ge1 + \delta. 
   \end{align}
\end{prop}

\begin{proof}
 We can assume that $X = \err^n$ with any given norm $\|\cdot\|.$ We
 will show that there exists $x \in S_X$ such that (\ref{eq:1}) holds
 for any positive $\delta \le (1 + n^{-1})^{1/2} - 1$. To
 prove this we use John's Theorem (see
 e.g. \cite[Theorem~6.22]{MR2766381}) that there exists an ellipsoid
 $J$ of minimal volume such that 
  \begin{align}
    \label{eq:2}
    n^{-1/2}J
  \subset B_X \subset J,
  \end{align}
  i.e. 
  \begin{align}
    \label{eq:3}
    \|x\|_J \le \|x\| \le n^{1/2} \|x\|_J,
  \end{align}
  where $\|\cdot\|_J$ denotes the Euclidian
  norm on $X$ we get when using $J$ as the unit ball. We have $S_X
  \cap \partial J \not = \emptyset$ where $\partial J = \{(x_i)_{i=1}^n \in
  X: \sum_{i=1}^n x_i^2 = 1\} \subset J$ since otherwise $J$ could not be of
  minimal volume. Take $x \in
  S_X \cap \partial J.$ Then for $h \in S_X$ we have
  \begin{align*}
    \max\|x \pm h\|^2
    &\ge \max\|x \pm h\|_J^2\\
    &= \max\{\|x\|_J^2 + \|h\|_J^2 \pm 2 \left<x,h \right>\}\\
    &\ge \|x\|_J^2 + (n^{-1/2}\|h\|)^2\\
    &\ge \|x\|_J^2 + n^{-1}\|h\|^2 = 1 + n^{-1}, 
  \end{align*}
   so $\max\|x \pm h\| \ge (1 + n^{-1})^{1/2} \ge 1 + \delta.$
\end{proof}

If you allow the dimensions of the finite dimensional spaces to grow,
the conclusion of Proposition \ref{prop:fixdim-john} is no longer
true. Indeed, we will construct finite dimensional spaces for which,
roughly speaking, the corners of the unit balls flattens out as the
dimensions grow. Let us introduce some notation and definitions
that we need. By $(e_j)_{j=1}^{\infty}$ denote, as usual, the standard unit vectors in
$c_{00}$ (the vector space of finitely supported vectors). Let $k, n,
m \in \enn$ with $k\le n \le m$ and set
\begin{align*}
  A_n&=\bigg\{\frac{1}{n}\sum_{j \in A} e_j^* \in \ell_1^m: A \subseteq \{1,
       \ldots, m\}, |A| = n \bigg\},\\
 B_k &= \bigg\{\frac{1}{k}e_j^* \in\ell_1^m: j = 1, \ldots, m \bigg\},
\end{align*}
where $\ell_1^m$ denotes the vector space $\err^m$ with the $\ell_1$-norm.
Put $C_{k,n}:= B_k \cup A_n.$ 
Define $\| \cdot \|_{k,n}: \ell_\infty^m \to \err$ by
\begin{align}
  \label{eq:5}
  \|f\|_{k,n}= \sup_{x \in C_{k,n}}|f(x)|,
\end{align}
where
\[
f(x)=
\begin{cases}
  n^{-1}\sum_{j \in \supp x}f_j, & x \in A_n,\\
  k^{-1}f_j,  & x= k^{-1}e_j^*\in B_k.
\end{cases}
\]  
Clearly $\|\cdot\|_{k,n}$ is a norm on $\ell_\infty^m$ satisfying
\begin{align}
  \label{eq:8}
  k^{-1}\|f\|_\infty \le \|f\|_{k,n} \le \|f\|_\infty.
\end{align}
Put $F_{k,n}:= (\ell_\infty^m, \|\cdot\|_{k,n}).$
\begin{lem}
 \label{lem:finitedim-ell-infty-1overk-lasq}
  Let $k, n, m \in \enn$ with $k^2\le n$ and $m=2n.$ Let $f \in F_{k,n}$ of norm
  $1.$ Then there is $h\in F_{k,n}$ of norm $1$ such that 
  \begin{align}
    \label{eq:4}
    \|f \pm h\|_{k,n} \le 1 + \frac{1}{k}.
  \end{align} 
 \end{lem}

 \begin{proof}
  Let $f \in S_{F_{k,n}}.$ We
  want to find $h \in S_{F_{k,n}}$ such that
  \[
  \|f \pm h\|_{k,n} \le 1 + \frac{1}{k}. 
  \]
  To this end note that by the pigeonhole principle there is a subset $E$ of $\{1, 2,
  \ldots, m\}$ with cardinality $\ge n$ on which $f$ has the
  same sign. Assume without loss of generality that $f$ is
  non-negative here. Since for any $x \in A_n$ with $\supp x \subseteq
  E$ we have 
  \begin{align}
    \label{eq:15}
    1 = \|f\|_{k,n} \ge  |f(x)| = \frac{1}{n}\sum_{j \in \supp x}|f_j|, 
  \end{align}
  there exists $l \in E$ such that $0 \le f_l \le 1.$ Now, let $h =
  ke_l.$ Then
  \begin{align}
    \label{eq:17}
    \|f\pm h\|_{k,n}
    \nonumber &=\max\{\sup_{x \in B_k}|(f\pm h)(x)|, \sup_{x \in C_n}|(f
      \pm h)(x)|\}\\
    \nonumber&=\max\bigg\{\sup_{j \in \enn}\frac{|f_j \pm
               ke_l(e_j^*)|}{k}, \sup_{x \in C_n}\frac{|\sum_{j \in
      \supp x}f_j\pm ke_l (e_j^*)|}{n}\bigg\}\\
     &\le\max\bigg\{\max\bigg\{\sup_{j \not=l}\frac{|f_j|}{k},
       \frac{|f_l| + k}{k}\bigg\}, \sup_{x \in
      C_n}\frac{|\sum_{j \in
      \supp x}f_j \pm ke_l(e_j^*)|}{n}\bigg\}\\
    &\nonumber\le \max\bigg\{ \frac{1 + k}{k}, \frac{|\sum_{j \in
      \supp x}f_j)|}{n} + \frac{k}{n}\bigg\}\\
    &\nonumber\le \max\bigg\{ 1 + \frac{1}{k}, 1 + \frac{k}{n}\bigg\}
      \le 1 + \frac{1}{k},
  \end{align}
   as wanted.
 \end{proof}
  
 \begin{rem}
   \label{rem:finitedim-ell-infty-1overk-asq}
   Note that we could get (\ref{eq:4}) in Lemma
   \ref{lem:finitedim-ell-infty-1overk-lasq} to
   hold for any given number $N$ of norm $1$ vectors simply by
   choosing the dimension $m = m(N)$ of the space $F_{k,n}$ to be sufficiently big.
 \end{rem}

\section{A bidual ASQ norm on $\ell_\infty$}
\label{sec:bidual-asq}
An alert reader will have noticed that the lower isomorphism constant in
(\ref{eq:8}) changes as the dimension grows. In order to construct an
equivalent bidual ASQ norm on $\ell_\infty$ we want to
avoid this. Thus we need to refine a bit the technique developed in
the previous section.

Fix $k\in \enn$ such that $k \ge 2.$ For $n \in \enn$ let $E_n =
\{2^n, \ldots, 2^{n+1} -1\}$ and $C_n =\{2^{-1}e_l^* \pm 2^{-1}e_m^*
\in \ell_1: l, m \in E_n, l \not = m\}.$ For $N \in \enn$ let 
\begin{align*}
  A_N&=\bigg\{\frac{1}{N}\sum_{j \in A} e_j^* \in \ell_1: A \subset
       \enn, |A| = kN \bigg\}\\
  B    &= \bigg\{\frac{1}{k}e_j^* \in \ell_1: j \in \enn \bigg\}.\\
  C   & = \cup_{n=1}^\infty C_n.
\end{align*}
Put $D_N:= C \cup B \cup A_N.$ 
Define $\| \cdot \|_N: \ell_\infty \to \err$ by
\begin{align}
  \label{eq:18}
  \|f\|_N= \sup_{x \in D_N}|f(x)|,
\end{align}
where 
\[
f(x)=
\begin{cases}
  N^{-1}\sum_{j \in \supp x}f_j, & x \in A_N,\\
  k^{-1}f_j,  & x= k^{-1}e_j^*\in B,\\
  2^{-1}f_l \pm 2^{-1}f_m, & x = 2^{-1}e_l^* \pm 2^{-1}e_m^* \in C,
\end{cases}
\]  
when $f=(f_j).$ Clearly $\|\cdot\|_N$ is a norm. Moreover, note that for $f \in
\ell_\infty$ we have
\begin{align}
  \label{eq:9}
  k^{-1}\|f\|_\infty \le \|f\|_N \le k \|f\|_\infty,
\end{align}
so $\|\cdot\|_N$ is an equivalent norm on $\ell_\infty$ for which the
isomorphism constants are fixed and independent of $N.$ Moreover, the norm
is bidual.

\begin{lem}
  \label{lem:normN-bidual}
  The space $(\ell_\infty,\|\cdot\|_N)$ is the bidual of $(c_0,\|\cdot\|_N).$
\end{lem}

\begin{proof}
  Note that $(e_j)_{j=1}^\infty$ is a basis for 
  $(c_0,\|\cdot\|_N)$ which is monotone and shrinking.
  Thus $(c_0,\|\cdot\|_N)^{**} $ can be identified with space of
  sequences $f = (f_j)_{j=1}^\infty$ such that $\|f\| =
  \sup_K\|\sum_{j=1}^Kf_je_j\|_N   < \infty.$ Now observe that 
 \begin{align*}
  \|f\|
  &= \sup_K\|\sum_{j=1}^K f_je_j\|_N  = \sup_N \|P_Kf\|_N\\
  &= \sup_K \sup_{x \in D_N}|P_Kf(x)| = \lim_K \sup_{x \in
    D_N}|P_Kf(x)| = \|f\|_N,   
 \end{align*}
 where $P_K$ denotes the projection on $\ell_\infty$ that projects
 vectors onto their first $K$ coordinates.
\end{proof}

Put $X_N:=(c_0, \|\cdot\|_N)$ and
$X^{**}_N:= (c_0, \|\cdot\|_N)^{**} = (\ell_\infty, \|\cdot\|_N).$ We
will need the following lemma.

\begin{lem}
  \label{lem:existnkm}
  For $N \in \enn, \eps > 0, $ and $(f^i)_{i=1}^K \subset
  S_{X^{**}_N},$ there exists $n \in \enn$ and distinct $l, m \in E_n$ such that
  \begin{enumerate}
  \item\label{lem:existnkm-1} $|f^i_j| \le k^{-1}$ for all $j \in E_n$
    and $i \in \{1, \ldots, K\},$
   \item\label{lem:existnkm-2}$|f^i_l - f^i_m| < \eps$ for all $i \in \{1, \ldots,
    K\}.$ 
  \end{enumerate}
\end{lem}

\begin{proof}
  Let $\eps > 0.$ Note that from the definiton of the norm each $f^i$
  has only finitely many coordinates  ($< 2kN$) of absolute
  value $> k^{-1}.$ Thus there exists $n_0 \ge 2$ such that
  \ref{lem:existnkm-1} holds for all $n \ge n_0.$ Now choose an integer
  $p > n_0$ such that $2k^{-1}/\eps < p.$ By the pigeonhole principle we
  can choose $n_1 > p$ such that $|f^1_l - f^1_m| < \eps$ for at least
  $p^K$ different pairs of distinct integers $l$ and
  $m$ in $E_{n_1}.$ Now again by the pigeonhole principle
  \ref{lem:existnkm-2} must hold for at least one of these pairs.  
\end{proof}

Let us prove that $X_N^{**}$ is getting closer and closer to being ASQ
as $N$ grows. Still it is not LASQ for any even number $N \ge k$.

\begin{lem}
  \label{lem:bidual-1overkN-asq}
  Let $N \in \enn$ and $(f^i)_{i=1}^K \subset S_{X^{**}_N}.$ Then there is
  $h\in S_{X_N}$ such that 
  \[
  \|f^i \pm h\|_N \le 1 + \frac{1}{N} 
  \]  
  for every $i =1, \ldots, K.$ Nevertheless, if $N$ is an even number
  $\ge k$, then $X_N$ is not LASQ.
\end{lem}

\begin{proof}
  Find $n \in \enn$ such that the conclusion in Lemma
  \ref{lem:existnkm} holds for $\eps = N^{-1}.$ Put $h = e_l - e_m.$
  Then $1 \ge \|h\|_N \ge h(2^{-1}e_l - 2^{-1}e_m) = 1,$ so $\|h\|_N = 1.$
  Note that
  \begin{align*}
    m_C
    &:=\sup_{x \in C}|(f^i \pm h) (x)|
     = \max \bigg\{\sup_{\substack{x \in C_n\\ \supp x \cap \{l,m\} \not =
      \emptyset}}|(f^i \pm h) (x)|, \sup_{\substack{x \in C \\ \supp x \cap
      \{l,m\} = \emptyset}}|f^i(x)|\bigg\}\\
    &\le \max \bigg\{\max|(f^i \pm h) (2^{-1}e_l^* \pm 2^{-1}e_m^*)|,
      \max_{\substack{j
      \in E_n\setminus\{l,m\} \\ p \in \{l,m\} } }|(f^i \pm h) (2^{-1}e_j^* \pm
      2^{-1}e_p^*)|, \|f^i\|_N \bigg\}\\
    & = \max \bigg\{\max|(2^{-1}(f^i_l \pm 1) \pm 2^{-1}(f^i_m \mp
      1)|, \max_{\substack{j \in E_n\setminus\{l,m\}\\ p \in \{l,m\} }
      }|(2^{-1}f^i_j \pm 2^{-1}(f^i_p \pm 1)|, 1 \bigg\}\\
    &\le \max \bigg\{\max\{k^{-1}, 1 + 2^{-1}\eps\}, k^{-1} +
      2^{-1}, 1 \bigg\}
      = 1 + 2^{-1}N^{-1},
  \end{align*}
  \begin{align*}
    m_B
    &:= \sup_{x \in B}\{|(f^i \pm h) (x)| = \max\{\sup_{j \not= l,
      m}k^{-1}|f^i_j|, k^{-1}\max_{j = l,m}|f^i_j \pm 1|\}\\
    &\le k^{-1}\max\{\|f^i\|_N, 1 + k^{-1}\} \le 1,
  \end{align*}
  and
  \begin{align*}
    m_{A_N}
    &:=\sup_{x \in A_N}|(f^i \pm h) (x)|\\
    &= \max \bigg\{ \sup_{\substack{x \in A_N\\ \supp x \cap\{l,m\} \not=
      \emptyset}} |(f^i \pm h) (x)|, \sup_{\substack{x \in A_N\\ \supp x \cap\{l,m\} =
       \emptyset}} |f^i(x)|\bigg\}\\
    &\le\max\{1 + N^{-1}, 1\} = 1 + N^{-1}.
  \end{align*}
  
  Thus
  \begin{align*}
    \|f^i \pm h\|_N
    & = \sup\{|(f^i \pm h)(x)|: x \in D_N\}\\
    & = \max\bigg\{\sup_{x \in C}|(f^i \pm h) (x)|, \sup_{x \in B}|(f^i \pm
      h) (x)|, \sup_{x \in A_N}|(f^i \pm h) (x)|\bigg\}\\
   & \le \max\{m_C, m_B, m_{A_N}\} \le 1 + N^{-1}
  \end{align*}
  for every $i = 1, \ldots, K.$

  The LASQ part: Let $E \subset \enn$ with $|E| = N/2$ such that $E \cap E_1 =
  \emptyset$ and $|E \cap E_n| = 1$ for all $2 \le n \le N/2 + 1$. Put
  $f = \sum_{j \in E}2e_j.$ Then $f$ has norm $1$. Let $0 < \eps < 1/3N$ and
  assume there is $h \in S_{X_N}$ such that $\|f \pm h\| \le 1 + \eps.$ We
  claim that for every $2 \le n \le N/2 +1,$ every $l, j \in E_n$ with
  $j \not= l$ and $l \in E_n \cap E,$ we have $|h_l| + |h_j| \le
  2\eps.$ In particular $|h_j| \le 2\eps.$
  Indeed, let $x = 2^{-1}(e_l^* \pm e_j^*) \in C_n.$ Then
  \begin{align*}
    1 + \eps 
    & \ge \max\|f \pm h\|\\
    &\ge \max(f \pm h)(x)\\
    &\ge 2^{-1}(2 + |h_l| + |h_m|) = 1 + 2^{-1}(|h_l| + |h_j|),
  \end{align*}
  so $|h_l| + |h_j| \le 2\eps$ and thus $|h_j| \le 2\eps.$ Hence for
  every $2 \le n \le N/2 + 1$ and every $x = 2^{-1}(e_l^* \pm e_m^*)
  \in C_n$ we have 
  \[
   |h(x)| \le 2^{-1}(|h_l| + |h_m|) \le 2\eps, 
  \]
  and for $x = k^{-1}e_j^*$ where $j \in E_n$ we have
 \[
  |h(x)| \le 2\eps k^{-1}.
  \]
  Also for $x = N^{-1}\sum_{j \in E} e_j^*$ we have
  \[
   |h(x)| 
   = |N^{-1}\sum_{j \in E} h_j| \le N^{-1}\sum_{j \in E} |h_j| 
   \le N^{-1}N2\eps/2 = \eps.
  \]
  Thus we have three possibilities;
  \begin{enumerate}
    \item \label{case1}there exists $y \in A_N$ such that $|h(y)| =
     |N^{-1}\sum_{j \in A}h_j| = 1,$ or 
   \item \label{case2} there exists $y \in B$ such that $|h(y)| =
     |h_j|/k = 1,$ or
   \item \label{case3}there exists $n \ge N/2 + 2$ and $y \in C_n$
     for which $h(y) = 2^{-1}(|h_l| + |h_m|) = 1.$ 
  \end{enumerate}
  Case \ref{case1}: We must
  have $|N^{-1}\sum_{j \in  A \setminus E} h_j| \ge 1 - \eps.$ Since $|A
  \setminus E| \ge (2k-1)N/2,$ there must exist a subset $A_1 \subset
  A \setminus E$ with $|A_1| = kN - N/2 = (2k -1)N/2$ such that
  $|N^{-1}\sum_{j \in A_1} h_j| \ge 1 - \eps - 1/(2k-1).$ Put $A_2 = E
  \cup A_1$ and $x = N^{-1}\sum_{j \in A_2} e_j^*.$ Then 
  \begin{align*}
   1 + \eps 
    &\ge\|f \pm h\| \\
    & \ge (f \pm h)(x)\\
    &\ge N^{-1}\max\bigg\{\sum_{j \in A_2} f_j + h_j, \sum_{j \in A_2}
      f_j - h_j\bigg\}\\
    & = N^{-1}\max\bigg\{\sum_{j \in E} f_j \pm \bigg(\sum_{j \in E}h_j +
      \sum_{j \in A_1}h_j\bigg)\bigg\}\\
    &\ge \max\bigg\{1 \pm N^{-1}\bigg(\sum_{j \in E}h_j +
      \sum_{j \in A_1}h_j\bigg)\bigg\}\\
    & \ge 1 -\eps + 1- \eps- 1/(2k-1),
  \end{align*}
   which implies that $\eps \ge 2/9 > 1/3N$, a contradiction.
  
  Case \ref{case2}: This is similar to Case \ref{case3}.
 
  Case \ref{case3}: Either $|h_l| \ge 1$ or $|h_m| \ge 1.$ Assume $|h_l| \ge 1.$
  Put $E_l = E \cup \{l\}$ and let $E_0 \subset \enn \setminus E_l$
  such that $A:= E_l \cup E_0$ satisfies $|A| = kN$ and $N^{-1}\sum_{j
    \in E_0} |h_j| < \eps.$ (This is possible as $h \in c_0.$) For $x
  := N^{-1}\sum_{j \in A}e_j^* \in A_N$ we get
  \begin{align*}
   1 + \eps 
    & \ge \|f \pm h\| \\
    &\ge (f \pm h)(x)\\
    &=N^{-1}\max\bigg\{\sum_{j \in A} f_j + h_j, \sum_{j
      \in A} f_j - h_j \bigg\}\\
    & = N^{-1}\max\bigg\{\sum_{j \in E} f_j \pm \bigg\{\sum_{j \in E} h_j + h_l +
      \sum_{j \in E_0} h_j\bigg)\bigg\}\\
    &= N^{-1}\max \bigg\{N \pm \bigg(\sum_{j \in E} h_j + h_l +
      \sum_{j \in E_0} h_j\bigg)\bigg\}\\
    &= \max\bigg\{1  \pm N^{-1}\bigg(\sum_{j \in E} h_j+ h_l +
      \sum_{j \in E_0} h_j\bigg)\bigg\}\\
    & \ge 1 - \eps + N^{-1} - \eps = 1 + N^{-1} - 2\eps,
  \end{align*}
  which implies that $\eps \ge 1/3N,$ a contradiction.
\end{proof}

Let $X:= c_0(X_N)$ where the $N$s are even numbers. Then $X^{**} =
\ell_\infty(X_N^{**}).$ Let $P_{\{N\}}:X^{**} \to X^{**}$ be the natural projection onto
$X_N^{**}$. Now we can finally prove our main result.

\begin{thm}
  \label{thm:bidual-ell-infty-asq}
  There exists a bidual renorming of $\ell_\infty$ which is ASQ.
\end{thm}
 
\begin{proof}
 The space $X^{**}$ satisfies the demands. Indeed, for the ASQ part
 let $\eps > 0$ and $(f^i)_{i=1}^K \subset S_{X^{**}}.$ Choose $M >
 1/\eps.$ From Lemma \ref{lem:bidual-1overkN-asq} there exists for all
 $i= 1, \ldots, K$ an $h \in
 S_{X_M}$ such that $\|P_{\{M\}}f^i \pm h\|_M \le 1 + 1/M.$ Thus 
 \begin{align*}
   \|f^i \pm h\| 
   & = \max\bigg\{\sup_{N \not=M}\|P_{\{N\}}f^i\|_N, \|P_{\{M\}}f^i
     \pm h\|_M\bigg\} \le 1 + \frac{1}{M} < 1 + \eps,
 \end{align*}
  as wanted.
  
  That $X^{**}$ is isomorphic to $\ell_\infty,$ follows since
  $X$ is isomorphic to $c_0.$ Indeed, define an operator $T: X \to
  c_0$ in the following way. For $x = (x_N) \in X$ let 
  \[
   T(x) = (x_1(1), x_2(1), x_1(2), x_2(2), x_3(1),\ldots). 
  \]  
   Clearly $T$ is a linear isomorphism since all the $\|\cdot\|_N$
   norms are equivalent with the same constants.
\end{proof}

\begin{rem}
  From the construction we see that there exists a bidual ASQ Banach
  space which is isomorphic to $\ell_\infty$ such that $k$ in (\ref{eq:9})
  is $= 2.$ Whether this constant is best possible we do not know.
\end{rem}

\begin{cor}
  \label{cor:x-not-membedded}
  The space $X$ is ASQ, but not M-embedded.
\end{cor}

\begin{proof}
  That $X$ is ASQ follows from the proof of Theorem
  \ref{thm:bidual-ell-infty-asq} (or one can
  simply use that a $c_0$ sum of Banach spaces is always
  ASQ). Moreover, all subspaces of an M-embedded space are M-embedded
  \cite[III.1.Theorem~1.6]{HWW}. Since by Lemma
  \ref{lem:bidual-1overkN-asq} not all component spaces $X_N$ are
  LASQ, and thus not M-embedded \cite[Theorem~4.2]{MR3415738}, $X$ is
  not M-embedded.  
\end{proof}

Let us also note that a similar argument to the one given in the proof of Lemma
\ref{lem:bidual-1overkN-asq}, shows that 
the first part of the conclusion of that lemma also holds for any
given $N \in \enn$ and any 
$(f_i)_{i=1}^K \subset S_{F_N}$ where $F_N:=(\ell^m_\infty, \|\cdot\|_N),$
provided the dimension $m = m(N)$ of $F_N$ is sufficiently
big. Note that every space $F_N$ is
isomorphic to $\ell_\infty^m$ with fixed constants $k^{-1}$ and $k$ (see
(\ref{eq:9})). Now if we let $Y = c_0(F_N)$ and use that the
property of being M-embedded is preserved by taking $c_0$-sums
\cite[III.1.Theorem~1.6]{HWW}, and otherwise argue as in the proof of Theorem
\ref{thm:bidual-ell-infty-asq}, we get that $Y$ satisfies the demands in  

\begin{thm}
   There exists a renorming of $c_0$ which is M-embedded and whose
   bidual is ASQ.  
\end{thm}

As consequences we now get the following renorming results. 

\begin{thm}
  \label{thm:dual-c0-asq-renorming}
  A dual Banach space admits a dual ASQ norm if and only if it contains
  an isomorphic copy of $c_0.$ 
\end{thm}

\begin{proof}
  If $(Y^*,\|\cdot\|)$ contains $c_0$, then $Y$ contains a
  complemented subspace $Z$ isomorphic to $\ell_1$ which in turn is
  isomorphic to $X^*$ where $X = c_0(X_N).$ Let $W$ be the
  complement of $Z$ in $Y.$ Then $Y$ is isomorphic to $W \oplus_1 X^*.
  $ Hence $Y^*$ is isomorphic to $(W \oplus_1 X^*)^* = W^*
  \oplus_\infty X^{**}$ which clearly is ASQ as $X^{**}$ is. 
\end{proof}

This improves \cite[Theorem~2.5]{MR3466077} which says that the result above
holds for ASQ norms (not necessarily dual norms).

\begin{thm}
  \label{thm:sepc0-bidualasq}
  Every separable Banach space which contains $c_0$ admits an
  equivalent norm such that its bidual (and thus the space itself) is ASQ.  
\end{thm}

\begin{proof}
  Let $Y$ be a separable Banach space which contains $c_0.$ By
  Sobczyk's theorem there exists a complemented subspace $V$ of $Y$
  which is isomorphic to $c_0.$ Let $Z$ be the complement of $V.$ Then
  \[
   Y = Z \oplus V \simeq Z \oplus_\infty c_0 \simeq Z \oplus_\infty X,
  \]
   where $X = c_0(X_N).$ Thus
  \[
   X^{**} \simeq (Y \oplus_\infty X)^{**} = Y^{**} \oplus_\infty X^{**},
   \] 
  which is ASQ.
\end{proof}

\begin{rem}
  From Theorem \ref{thm:sepc0-bidualasq} we have in particular that every
  separable ASQ space admits an equivalent norm for which the bidual
  is ASQ. This result is in analogy with the recent result from
  \cite{langemets_lopez-perez} (partly answering a question of Godefroy) that
  every separable octahedral Banach spaces admits an equivalent norm
  for which the bidual is octahedral (see Section
  \ref{sec:char-sep-asq} for the definition of an octahedral space).   
\end{rem}

\section{A characterization of separable real ASQ spaces}
\label{sec:char-sep-asq}
Let us end the paper with a characterization of separable real ASQ
spaces. The characterization is in the same vein as following result of
Godefroy from \cite[p.~12]{G-octa}.
\begin{thm}
  Let $X$ be a separable Banach space. Then the following are
  equivalent.
  \begin{enumerate}
    \item $X$ is octahedral.
    \item There exists $h \in X^{**} \setminus \{0\}$ such that
      for all $x \in X$ we have
      \begin{align*}
        \|x + h\| = \|x\| + \|h\|.
      \end{align*}
  \end{enumerate}
\end{thm}
Recall that a Banach space is octahedral if for every for $\eps > 0$
and every finite set $(x_i)_{i=1}^n \subset S_X$ there is $h \in S_X$
such that $\|x_i \pm h\| \ge 2 - \eps$ for every $i=1, \ldots, n.$ 
\begin{prop}\label{prop:asq-char}
  Let $X$ be a separable real Banach space. Then the following are equivalent.
  \begin{enumerate}
    \item $X$ is ASQ.
    \item \label{b}There exists $h\in X^{****}\setminus \{0\}$ such
      that for all $x \in X$
  \begin{align*}
   \|x + h\| = \max\{\|x\|, \|h\|\}. 
  \end{align*}
  \end{enumerate}
\end{prop}

\begin{rem}
  Statement \ref{b} says that $X$ is an M-summand in the subspace $\spann(X,
  \{h\})$ in the fourth dual $X^{****}$ of $X$ (see \cite{HWW} for
  a definition of an M-summand). Under the assumption that $X$ is ASQ,
  it is generally not possible to find non-trivial subspaces in
  the second dual $X^{**}$ in which $X$ is an M-summand. Indeed, if
  this was possible
  it would follow that there exists $h\in X^{**}\setminus \{0\}$ such
  that $\|h + x\| = \|h - x\|$ for all $x \in X.$ But this yields that
  $X$ contains $\ell_1$
  by a result of Maurey \cite[Theorem]{Mau}. Thus
  $X = c_0$ is a concrete example where $X$ is ASQ and where \ref{b}
  with $h\in X^{**}\setminus \{0\}$ fails. The
  upshot of Proposition \ref{prop:asq-char}
  is, however, that for a real separable ASQ space $X,$ there is always
  enough room to find subspaces in the fourth dual in which $X$ is an M-summand.
\end{rem}

To prove Proposition \ref{prop:asq-char} we need a lemma.
\begin{lem}\label{prop:ASQ=superASQ}
  In a separable ASQ space there exists a sequence $(h_n) \subset
  S_X$ such that for all $x \in X$
  \begin{align*}
    \lim_n\|x+h_n\| = \max\{\|x\|, 1\}. 
  \end{align*}
\end{lem}

\begin{proof}
  Let $(x_i)_{i=1}^\infty$ be a dense subset of $S_X$ and $(\eps_n)$ a
  decreasing $0$ sequence of positive numbers. Since $X$ is ASQ we
  can for each $n \in \enn$ find $h_n \in S_X$ such
  that
  \begin{align}\label{ineq1}
    |\|x_i \pm h_n\| -1| < \eps_n \text{ for every } i = 1, \ldots, n.
  \end{align}
  
  Then for all $x_i$ we have
  \[\|x_i \pm h_n\| \to 1.\]
  Now let $x \in S_X$ arbitrarily and let $\eps > 0$. Find $x_i$ such
  that $\|x - x_i\| < \eps/2$ and $N \in \enn$ (using (\ref{ineq1}))
  such that $\eps_n < \eps/2$ for all $n \ge N$. Then for all $n \ge N$ we have
  \begin{align*}
    \|x \pm h_n\| &\le \|x - x_i\| + \|x_i - h_n\| < 1 + \eps \text{ and
    }\\
    \|x \pm h_n\| &\ge \|x_i - h_n\| - \|x - x_i\| > 1 - \eps,
  \end{align*}
 and we are done.
 \end{proof}

\begin{proof}[Proof of Proposition~\ref{prop:asq-char}]
  (a) $\Rightarrow$ (b). Let $(h_n) \subset S_X$ be the sequence from
  Proposition \ref{prop:ASQ=superASQ}. Note that the function $\tau: X
  \to \err$ given by
  \begin{align*}
    \tau(x) = \lim_n\|x+h_n\|
  \end{align*}
  defines a type with the property that $\tau(x) = \max\{\|x\|,1\}.$
  Using this fact it is straightforward to show that $\tau$ is a $c_{0+}$-type, i.e.
  \begin{align*}
    \tau(x) = \lim_m\lim_n\|x + ah_m +bh_n\|
  \end{align*}
  whenever $a,b \ge 0$ and $\max\{a,b\} = 1.$ (Actually it holds for
  any $a,b \in \err$ with $\max\{|a|, |b|\}=1$ so $\tau$ is an
  $\ell_\infty$-type.) Thus by \cite[Proposition~2.10]{MR920994},
  there exists $h \in X^{****}$ so that for all $x \in X$
  \begin{align*}
    \tau(x) = \|x+ h\|. 
  \end{align*}
 It follows that for all $x$ in $X$ we have $\|x + h\| = \max\{\|x\|, 1\}$.

(b) $\Rightarrow$ (a). This is straightforward using the Principle of
Local Reflexivity twice.
\end{proof}

\section*{Acknowledgement}
The authors would like to thank Vegard Lima for a conversation that
led to the question of whether Theorem \ref{thm:sepc0-bidualasq} could
be true.

\newcommand{\etalchar}[1]{$^{#1}$}
\def\cprime{$'$} \def\cprime{$'$} \def\cprime{$'$}
\providecommand{\bysame}{\leavevmode\hbox to3em{\hrulefill}\thinspace}
\providecommand{\MR}{\relax\ifhmode\unskip\space\fi MR }
\providecommand{\MRhref}[2]{%
  \href{http://www.ams.org/mathscinet-getitem?mr=#1}{#2}
}
\providecommand{\href}[2]{#2}

\end{document}